\documentclass{amsart}
\usepackage{amssymb}

\numberwithin{equation}{section}
\newtheorem{thm}{Theorem}[section]
\newtheorem{lem}[thm]{Lemma}
\newtheorem{prop}[thm]{Proposition}

\newtheorem{conj}[thm]{Conjecture}
\theoremstyle{definition}

\newtheorem{prob}[thm]{Problem}
\theoremstyle{remark}
\newtheorem{rmk}[thm]{Remark}

\title{Conjectures on tilting modules and antispherical $p$-cells}

 \author{Pramod N. Achar}
 \address{Department of Mathematics\\
   Louisiana State University\\
   Baton Rouge, LA 70803\\
   U.S.A.}
 \email{pramod@math.lsu.edu}
 
  \author{William Hardesty}
   \address{Department of Mathematics\\
   Louisiana State University\\
   Baton Rouge, LA 70803\\
   U.S.A.}
  \email{whardesty@lsu.edu}
 
 \author{Simon Riche}
 \address{Universit\'e Clermont Auvergne, CNRS, LMBP, F-63000 Clermont-Ferrand, France.
 }
 \email{simon.riche@uca.fr}
 
 \thanks{P.A. was supported by NSF Grant No.~DMS-1802241. This project has received funding from the European Research Council (ERC) under the European Union's Horizon 2020 research and innovation programme (grant agreement No. 677147).}

\usepackage{mathrsfs}
\usepackage{hyperref}

\newcommand{\bk}{\Bbbk}
\newcommand{\Z}{\mathbb{Z}}
\newcommand{\C}{\mathbb{C}}

\newcommand{\Rep}{\mathrm{Rep}}
\newcommand{\Tilt}{\mathrm{Tilt}}
\newcommand{\Fr}{\mathrm{Fr}}

\newcommand{\tilt}{\mathsf{T}}
\newcommand{\irr}{\mathsf{L}}

\newcommand{\red}{{\mathrm{red}}}
\newcommand{\unip}{{\mathrm{unip}}}

\newcommand{\sG}{\mathsf{G}}
\newcommand{\sg}{\mathsf{g}}
\newcommand{\fg}{\mathfrak{g}}

\newcommand{\cohsgb}{\mathsf{H}_{\sg}^\bullet}
\newcommand{\Tsim}{\sim^{\mathsf{T}}}

\newcommand{\LVsim}{\sim^{\mathrm{LV}}}
\newcommand{\SL}{\mathrm{SL}}
\newcommand{\GL}{\mathrm{GL}}
\newcommand{\sa}{\mathsf{a}}
\newcommand{\mult}{\mathsf{mult}}
\newcommand{\cP}{\mathcal{P}}
\newcommand{\reg}{{\mathrm{reg}}}
\newcommand{\sreg}{{\mathrm{sreg}}}

\newcommand{\bX}{\mathbf{X}}

\newcommand{\bY}{\mathbf{Y}}
\newcommand{\Wext}{W_{\mathrm{ext}}}
\newcommand{\Wf}{W_{\mathrm{f}}}

\newcommand{\fWext}{{}^{\mathrm{f}} \hspace{-1pt} W_{\mathrm{ext}}}

\newcommand{\fWfext}{ {}^{\mathrm{f}} \hspace{-1pt} \Wext^{\mathrm{f}}}
\newcommand{\lcalc}{\underline{\mathrm{Alc}}}

\newcommand{\Haff}{\mathcal{H}}
\newcommand{\Hext}{\mathcal{H}_{\mathrm{ext}}}
\newcommand{\Hf}{\mathcal{H}_{\mathrm{f}}}
\newcommand{\puH}{{}^p \hspace{-1pt} \underline{H}}

\newcommand{\puN}{{}^p \hspace{-1pt} \underline{N}}
\newcommand{\Masph}{\mathcal{M}_{\mathrm{asph}}}
\newcommand{\sgn}{\mathrm{sgn}}

\newcommand{\cN}{\mathcal{N}}
\newcommand{\scO}{{\mathscr{O}}}
\newcommand{\Coh}{\mathrm{Coh}}
\newcommand{\scL}{\mathscr{L}}
\newcommand{\scT}{\mathscr{T}}

\newcommand{\bb}{\mathbf{b}}
\newcommand{\bc}{\mathbf{c}}

\DeclareMathOperator{\Ext}{Ext}

\newcommand{\simto}{\xrightarrow{\sim}}
\newcommand{\simbij}{\overset{\sim}{\longleftrightarrow}}
\newcommand{\la}{\langle}
\newcommand{\ra}{\rangle}
\newcommand{\supp}{\mathrm{supp}}

\newcommand{\Waff}{W}

\makeatletter
\def\lotimes{\@ifnextchar_{\@lotimessub}{\@lotimesnosub}}
\def\@lotimessub_#1{\mathchoice{\mathbin{\mathop{\otimes}^L}_{#1}}%
  {\otimes^L_{#1}}{\otimes^L_{#1}}{\otimes^L_{#1}}}
\def\@lotimesnosub{\mathbin{\mathop{\otimes}^L}}
\makeatother

\begin{document}

\begin{abstract}
For quantum groups at a root of unity, there is a web of theorems (due to Bezrukavnikov and Ostrik, and relying on work of Lusztig) connecting the following topics: (i)~tilting modules; (ii)~vector bundles on nilpotent orbits; and (iii)~Kazhdan--Lusztig cells in the affine Weyl group.  In this paper, we propose a (partly conjectural) analogous picture for reductive algebraic groups over fields of positive characteristic, inspired by a conjecture of Humphreys.
\end{abstract}

\maketitle

\section{Introduction}
\label{sec:intro}

Let $\sG$ be a connected reductive algebraic group over an algebraically closed field $\bk$ of characteristic $p$, with simply-connected derived subgroup.  Assume that $p$ is larger than the Coxeter number $h$ for $\sG$.  The goal of this paper is to present some conjectures relating the following three topics: (i)~tilting $\sG$-modules; (ii)~coherent sheaves on the nilpotent cone; and (iii)~Kazhdan--Lusztig cells (or $p$-cells) in the affine Weyl group.  These conjectures can be seen as an attempt to put the Humphreys conjecture  on support varieties of tilting modules (\cite{hum:cmr}) in a bigger picture (involving, in particular, the Lusztig--Vogan bijection), in the hope of proving it in full generality.

There is a parallel story in characteristic $0$, in which the reductive group $\sG$ is replaced by a quantum group at a root of unity.  In this setting, the analogues of our conjectures are theorems, mostly due to Bezrukavnikov and Ostrik, cf.~\cite{bezru2} and \cite{ostrik} respectively, and rely on work of Lusztig (see in particular~\cite{lusztig}).
Our conjectures are inspired by these results (which we will review below), as well as by work of Andersen~\cite{hha:pcells} in the positive characteristic setting.

\section{Weyl groups, Hecke algebras, and cells}

\subsection{Settings}

Let $R = (\bX, \Phi, \bY, \Phi^\vee)$ be a root datum, and let $h$ be the Coxeter number for $\Phi$. We will assume that $\bY/\Z\Phi^\vee$ has no torsion. Fix a set of positive roots $\Phi^+ \subset \Phi$, and let $\bX^+ \subset \bX$ be the corresponding set of dominant weights.  We will consider the following two settings:
\begin{enumerate}
\item (\textbf{The quantum case}) Let $\bk$ be an algebraically closed field of characteristic $0$, and let $\sG$ be the Lusztig quantum group associated to $R$ over $\bk$, specialized at an $\ell$th root of unity.  Here, we assume that $\ell$ is odd, that $3 \nmid \ell$ if $\Phi$ contains a factor of type $\mathbf{G}_2$, and that $\ell > h$.
\item (\textbf{The reductive case}) Let $\bk$ be an algebraically closed field of characteristic $\ell > h$, and let $\sG$ be the connected reductive algebraic group over $\bk$ with root datum $R$.
\end{enumerate}
In the quantum case, one may take $\bk = \C$, of course.  It will be convenient to have a separate notation for the characteristic of $\bk$.  We set
\[
p = \mathop{\mathrm{char}} \bk =
\begin{cases}
0 & \text{in the quantum case,} \\
\ell & \text{in the reductive case.}
\end{cases}
\]
In both cases, the isomorphism classes of indecomposable tilting $\sG$-modules (of type I in the quantum case) are naturally parametrized by $\bX^+$.  For $\lambda \in \bX^+$, let
\[
\tilt(\lambda)
\]
be the indecomposable tilting $\sG$-module with highest weight $\lambda$.

\subsection{Weyl groups}

Let $\Wf$ be the (finite) Weyl group of $\Phi$, and let $\Wext = \Wf \ltimes \bX$ be the (extended) affine Weyl group. For $\lambda \in \bX$, we write $t_\lambda$ for the corresponding element of $\Wext$. 
The subgroup $\Waff := \Wf \ltimes \Z\Phi \subset \Wext$ admits a natural Coxeter group structure, whose length function is given by
\[
   \ell(t_\lambda x) = \sum_{\substack{\alpha \in \Phi^+ \\ x(\alpha) \in \Phi^+}} |\langle \lambda, \alpha^\vee \rangle| + \sum_{\substack{\alpha \in \Phi^+ \\ x(\alpha) \in -\Phi^+}} |1+\langle \lambda, \alpha^\vee \rangle|
\]
for $x \in \Wf$ and $\lambda \in \Z\Phi$. This formula allows one to extend this function to $\Wext$. Moreover, if we set $\Omega := \{w \in \Wext \mid \ell(w)=0\}$, then conjugation by elements of $\Omega$ induces Coxeter group automorphisms of $\Waff$, and multiplication induces a group isomorphism $\Omega \ltimes \Waff \simto \Wext$.

For $\lambda \in \bX$, we set
\[
w_\lambda = \text{the unique element of minimal length in the right coset $\Wf t_\lambda \subset \Wext$,}
\]
and let $\fWext = \{ w_\lambda : \lambda \in \bX \}$.  This is precisely the set of minimal-length representatives for the set of right cosets $\Wf \backslash \Wext$.  In other words, the assignment $\lambda \mapsto w_\lambda$ gives a bijection
\begin{equation}\label{eqn:fw-param}
\bX \simto \fWext \cong \Wf \backslash \Wext.
\end{equation}

Next, let $\fWfext$ be the set of minimal-length coset representatives for the double quotient $\Wf \backslash \Wext / \Wf$.  This is a subset of $\fWext$; more precisely we have $w_\lambda \in \fWfext$ if and only if $\lambda \in - \bX^+$.  In other words,~\eqref{eqn:fw-param} restricts to a bijection
\begin{equation}\label{eqn:fwf-param}
-\bX^+ \simto \fWfext \cong \Wf \backslash \Wext / \Wf.
\end{equation}

We define the ($\ell$-dilated) ``dot action'' of $\Wext$ on $\bX$ as follows: for $w = v \cdot t_\lambda$, where $v \in \Wf$ and $\lambda \in \bX$, we set
\[
w \cdot_\ell \mu = v(\mu + \ell\lambda + \rho) - \rho,
\]
where, as usual, $\rho = \frac{1}{2} \sum_{\alpha \in \Phi^+} \alpha$.  It is well known that $w \cdot_\ell 0 \in \bX^+$ if and only if $w \in \fWext$. 

\subsection{Hecke algebras and cells}

Let $\Hext$ be the Hecke algebra attached to $\Wext$, i.e.~the free $\Z[v,v^{-1}]$-module with a basis $\{H_w : w \in \Wext\}$ and multiplication determined by the following rules:
\begin{gather*}
 (H_s + v)(H_s - v^{-1})=0 \quad \text{if $\ell(s)=1$;}\\
 H_x \cdot H_y = H_{xy} \quad \text{if $\ell(xy)=\ell(x)+\ell(y)$.}
\end{gather*}
Then the subalgebra $\Hf$, resp.~$\Haff$, of $\Hext$ spanned by the elements $H_w$ with $w \in \Wf$, resp.~$\Waff$, identifies with the Hecke algebra of the Coxeter group $\Wf$, resp.~$\Waff$. At $v=1$, the algebra $\Hext$, resp.~$\Haff$, resp.~$\Hf$, specializes to the group algebra of $\Wext$, resp.~$\Waff$, resp.~$\Wf$.

Let $\sgn$ denote the ``sign representation'' of $\Hf$, i.e.~the $\Hf$-module structure on $\Z[v,v^{-1}]$ for which $H_s$ acts by multiplication by $-v$ if $s \in \Wf$ and $\ell(s)=1$. (This module specializes at $v = 1$ to the usual sign representation of $\Wf$.)  Recall that the \emph{antispherical module} is the right $\Hext$-module given by
\[
\Masph = \sgn \otimes_{\Hf} \Hext.
\] 

Following~\cite{jw,rw}, the algebra $\Hext$ admits a basis $( \puH_w : w \in \Wext )$, known as the \emph{$p$-canonical basis}; these elements satisfy
\begin{equation}
\label{eqn:puH-Wext}
 \puH_{w\omega} = \puH_w \cdot H_\omega \quad \text{and} \quad \puH_{\omega w} = H_\omega \cdot \puH_w
\end{equation}
for $\omega \in \Omega$ and $w \in \Waff$.
In the quantum case, i.e., when $p = 0$, this basis goes back to~\cite{kl}, and is also known as the \emph{Kazhdan--Lusztig basis}.  In $\Masph$, we have $1 \otimes \puH_w = 0$ if and only if $w \notin \fWext$.  Let
\[
\puN_w = 1 \otimes \puH_w \in \Masph
\qquad\text{for $w \in \fWext$.}
\]
The set $(\puN_w : w \in \fWext )$ forms a $\Z[v,v^{-1}]$-basis for $\Masph$, again known as the \emph{$p$-canonical basis of the antispherical module}.

An $\Hext$-submodule $M' \subset \Masph$ is called a \emph{$p$-based submodule} if it is spanned (as a $\Z[v,v^{-1}]$-module) by the $p$-canonical basis elements that it contains.  For any $m \in \Masph$, there is a unique minimal $p$-based submodule containing $m$; we call this submodule the \emph{principal $p$-based submodule generated by $m$}.  (In general, this submodule is much larger than the usual submodule $m \Haff$ generated by $m$.)  We define an equivalence relation on $\fWext$ as follows:
\[
w \sim^p v
\qquad
\begin{array}{@{}c@{}}
\text{if $\puN_w$ and $\puN_v$ generate the same}\\
\text{principal $p$-based submodule of $\Masph$.}
\end{array}
\]
Equivalence classes for $\sim^p$ are called \emph{antispherical $p$-cells}.

One can similarly define $p$-based left, right, or two-sided ideals in $\Hext$, and consider the principal $p$-based ideals generated by an element $\puH_w$ in $\Hext$.  This leads to the notions of \emph{left}, \emph{right}, and \emph{two-sided $p$-cells} in $\Wext$. In this way, one can interpret the antispherical $p$-cells as considered above as the right $p$-cells which intersect $\fWext$ (or, equivalently, are contained in $\fWext$); see e.g.~\cite[\S 5.3]{ahr1}.

Again, when $p = 0$, these notions go back to~\cite{kl}, and are often simply called \emph{Kazhdan--Lusztig cells}. The main result of~\cite{lx} implies that every two-sided $0$-cell in $\Wext$ meets $\fWext$ in exactly one antispherical $0$-cell.  Thus, there is a canonical bijection
\begin{equation}\label{eqn:2sided}
\{ \text{two-sided $0$-cells} \} \simbij \{ \text{antispherical $0$-cells} \}.
\end{equation}

\begin{rmk}
\label{rmk:cells-Wext}
 The theory of cells (or $p$-cells) is usually mostly considered in the setting of Coxeter groups (see~\cite{kl,jensen}); in the present setting this would mean replacing $\Hext$ by $\Haff$ (with its basis $( \puH_w : w \in \Waff )$), and the antispherical module $\Masph$ by the $\Haff$-submodule $\Masph'=\sgn \otimes_{\Hf} \Haff$ (with its basis $(\puN_w : w \in \fWext \cap \Waff)$). However one can check using~\eqref{eqn:puH-Wext} that our relation $\sim^p$ restricts to the ``usual'' equivalence relation on $\Waff$.  Moreover, one can see that for $w,w' \in \Waff$ and $\omega, \omega' \in \Omega$ we have
 \[
  w\omega \sim^p w'\omega' \quad \Leftrightarrow \quad w \sim^p w'.
 \]
In particular, antispherical cells are stable under right multiplication by elements in $\Omega$, and intersecting with $\Waff$ induces a bijection between the set of antispherical $p$-cells for $\Wext$ and the set of ``usual'' antispherical $p$-cells for the Coxeter group $\Waff$. Similar comments apply to the other sorts of cells considered above.
\end{rmk}

\subsection{Nilpotent orbits and the Lusztig--Vogan bijection}
\label{ss:nilpotent}

Let $\dot G$ be the connected reductive group over $\bk$ with root datum $(\ell\bX, \ell\Phi, \frac{1}{\ell}\bY, \frac{1}{\ell} \Phi^\vee)$.  Of course, this root datum is isomorphic to $R$, but the weight lattice has been scaled by $\ell$.  In the reductive case, $\dot G$ can be identified with the Frobenius twist of $\sG$.  In the quantum case, $\dot G$ plays a similar conceptual role: the universal enveloping algebra of its Lie algebra $U(\dot\fg)$ is the target of Lusztig's quantum Frobenius map.

Let $\cN$ be the nilpotent cone in the Lie algebra of $\dot G$.  It is well known that $\dot G$ acts on $\cN$ with finitely many orbits.  The following remarkable theorem tells us, in particular, that the set of antispherical $0$-cells is also finite.

\begin{thm}[Lusztig~\cite{lusztig}]\label{thm:lusztig}
There is a canonical bijection
\[
\begin{array}{r@{}c@{}l}\{\text{antispherical $0$-cells}\} &{}\simbij{}& \{ \text{$\dot G$-orbits on $\cN$} \} \\
\bc &\mapsto& \scO_\bc
\end{array}
\]
\end{thm}

This result was stated in~\cite{lusztig} in terms of two-sided cells in $\Waff$ rather than antispherical cells for $\Wext$, but we may rephrase it this way using~\eqref{eqn:2sided} and Remark~\ref{rmk:cells-Wext}.  The statement in~\cite{lusztig} also assumes that $\bk = \C$.  However, our assumptions on $\bk$ imply that $p$ is at least good for $\dot G$.  It is well known that in good characteristic, the set of nilpotent orbits admits a combinatorial parametrization (via the Bala--Carter theorem) that depends only on the root datum.  In other words, the set of nilpotent orbits is independent of $p$.  Thus, Theorem~\ref{thm:lusztig} remains valid for all $p$ considered here.

In~\cite{lusztig}, Lusztig conjectured that the bijection of Theorem~\ref{thm:lusztig} could be refined in a way that involves $\fWfext$ and vector bundles on nilpotent orbits.  This refined bijection, which was independently conjectured by Vogan~\cite{vogan} from a different perspective and has been established by Bezrukavnikov, is known as the \emph{Lusztig--Vogan bijection}.  To state this bijection, we need some additional notation.

Given an orbit $\scO \subset \cN$, let $\Coh^{\dot G}(\scO)$ be the abelian category of $\dot G$-equivariant coherent sheaves on $\scO$.  Of course, every $\dot G$-equivariant coherent sheaf on $\scO$ is automatically a vector bundle, and hence an extension of irreducible vector bundles.  Let
\[
\Sigma_\scO = 
\begin{array}{@{}c@{}}
\text{the set of isomorphism classes of irreducible} \\
\text{$\dot G$-equivariant vector bundles on $\scO$}.
\end{array}
\]
For $\sigma \in \Sigma_\scO$, denote the corresponding irreducible vector bundle by
\[
\scL_\scO(\sigma).
\]
Finally, let
\[
\Xi = \{ (\scO,\sigma) \mid \text{$\scO \subset \cN$ a $\dot G$-orbit, and $\sigma \in \Sigma_\scO$} \}.
\]

\begin{thm}[Lusztig--Vogan bijection~\cite{bezru1, bezru2, achar, ahr3}]\label{thm:lv-bij}
\ 
\begin{enumerate}
\item There is a canonical bijection\label{it:lvbij}
\[
\begin{array}{r@{}c@{}l}
- \bX^+ &{}\simbij{}& \Xi \\
\lambda & \mapsto & (\scO_\lambda, \sigma_\lambda)
\end{array}
\]
\item The bijection of part~\eqref{it:lvbij} is compatible with Theorem~\ref{thm:lusztig} in the following way: for $\lambda \in -\bX^+$, we have\label{it:lvcompat}
\[
\scO_\lambda = \scO_\bc,
\]
where $\bc \subset \fWext$ is the antispherical $0$-cell containing $w_\lambda$.
\end{enumerate}
\end{thm}

Here is a brief history of this result.
In~\cite{bezru1}, Bezrukavnikov proved part~\eqref{it:lvbij} for $\bk = \C$, using the theory of \emph{perverse-coherent sheaves}.
In~\cite{bezru-cells} he obtained a different construction of such a bijection, for which part~\eqref{it:lvcompat} is clear, and in~\cite{bezru3} he proved that the two constructions give identical bijections.
In~\cite{achar}, the first author showed that the arguments of~\cite{bezru1} can be adapted to positive characteristic, establishing part~\eqref{it:lvbij} for all $\bk$. Of course, in this bijection the left-hand side only depends on the root datum $R$, but the right-hand side involves the field $\bk$.
In~\cite{ahr3}, the authors showed that the bijection
is, in a suitable sense, ``independent of $\bk$.'' In particular, for any $\lambda \in -\bX^+$, the orbits $\scO_\lambda$ for all fields $\bk$ match under the Bala--Carter bijection. As a consequence, part~\eqref{it:lvcompat} holds for all $\bk$.

\section{Cohomology and support varieties}

\subsection{Frobenius kernel cohomology}
\label{ss:Fr-cohomology}

Let $\sg$ denote the following object:
\begin{enumerate}
\item In the quantum case, $\sg$ is the small quantum group associated to $R$ at an $\ell$th root of unity.
\item In the reductive case, $\sg$ is the (scheme-theoretic) kernel of the Frobenius map $\Fr: \sG \to \dot G$.
\end{enumerate}
Consider the trivial $\sg$-module $\bk$.  In both cases, there is a $\dot G$-equivariant ring isomorphism
\begin{equation}\label{eqn:cohom}
\Ext^\bullet_{\sg}(\bk,\bk) \cong \bk[\cN].
\end{equation}
In the reductive case, this is due to Andersen--Jantzen~\cite{aj} (and earlier, under slightly  more restrictive hypotheses, to Friedlander--Parshall~\cite{fp}). In the quantum case, this is due to Ginzburg--Kumar~\cite{gk} (see also Bendel--Nakano--Parshall--Pillen~\cite{bnpp:qggnc}).

Thanks to~\eqref{eqn:cohom}, for any two finite-dimensional $\sG$-modules $M$ and $N$, the space $\Ext^\bullet_{\sg}(M,N)$ acquires the structure of a finitely generated $\dot G$-equivariant $\bk[\cN]$-mod\-ule, or, equivalently, a $\dot G$-equivariant coherent sheaf on $\cN$.  Given a $\sG$-module $M$, we define its \emph{support variety} to be the closed subset of $\cN$ given by
\[
V_\sg(M) = \text{support of the coherent sheaf $\Ext^\bullet_{\sg}(M,M)$.}
\]
The \emph{cohomology} of $M$, denoted by $\cohsgb(M)$, is defined by
\[
\cohsgb(M) = \Ext^\bullet_{\sg}(\bk,M), \qquad\text{regarded as an object of $\Coh^{\dot G}(\cN)$.}
\]
The support of $\cohsgb(M)$ is called the \emph{relative support variety} of $M$.

We are primarily interested in the cohomology and support varieties of tilting modules. For these a conjectural description has been given by Humhreys~\cite{hum:cmr}; this description is now known to be correct in some cases, as explained below. Before explaining that we make a few preliminary remarks. It is straightforward to check that
\[
\cohsgb(\tilt(\lambda)) = 0 \qquad\text{if $\lambda \notin \fWfext \cdot_\ell 0$.}
\]
Via~\eqref{eqn:fwf-param}, $\fWfext \cdot_\ell 0$ is in bijection with $-\bX^+$.  Thus, we must study
\[
\cohsgb(\tilt(w_\mu \cdot_\ell 0)) \qquad\text{for}\qquad \mu \in -\bX^+.
\]

\subsection{Alcoves}

The study of support varieties of tilting modules requires the notion of \emph{alcoves}.  Recall that an \emph{alcove} is a non-empty subset of $\bX$ of the form
\[
\{ \lambda \in \bX \mid  \text{$n_\alpha\ell  < \langle \lambda + \rho, \alpha^\vee \rangle < (n_\alpha+1)\ell$ for all $\alpha \in \Phi^+$} \},
\]
where $(n_\alpha : \alpha \in \Phi^+)$ is some set of integers. The \emph{lower closure} of such an alcove
is defined by relaxing the left inequality to $\le$.  
Then every weight $\lambda \in \bX$ belongs to a unique lower closure of an alcove,
and if for $w \in \Wext$ we denote by $\lcalc(w)$
the lower closure of the alcove containing $w \cdot_\ell 0$, then the map $w \mapsto \lcalc(w)$
induces a bijection
\[
\Wext / \Omega  \xrightarrow{\sim}  \{ \text{lower closures of alcoves} \}.
\]
The lower closure $\lcalc(w)$ meets $\bX^+$ if and only if $w \in \fWext$. We also note that 
$V_{\sg}(\tilt(\lambda)) = V_{\sg}(\tilt(\mu))$ for any $w \in \fWext$ and $\lambda, \mu \in \lcalc(w)$ (see~\cite[Proposition~8]{hha:pcells} or~\cite[Proposition~3.1.3]{hardesty}).

\subsection{The quantum case}
\label{ss:cohom-q}

The main results about the cohomology and support varieties of tilting modules in the quantum case are the following.

\begin{thm}[\cite{bezru2}]\label{thm:bezru-coh}
Suppose we are in the quantum case. Let $\mu \in -\bX^+$, and
let $(\scO_\mu, \sigma_\mu) \in \Xi$ be the pair corresponding to $\mu$ under the Lusztig--Vogan bijection (see~\S\ref{ss:nilpotent}). The coherent sheaf $\cohsgb(\tilt(w_\mu \cdot_\ell 0))$ is (scheme-theoretically) supported on $\overline{\scO_\mu}$, and
\[
\cohsgb(\tilt(w_\mu \cdot_\ell 0))|_{\scO_\mu} \cong \scL_{\scO_\mu}(\sigma_\mu).
\]
\end{thm}

The proof of this theorem gives somewhat finer information about $\cohsgb(\tilt(w_\mu \cdot_\ell 0))$: it is the total cohomology of a certain simple perverse-coherent sheaf on $\cN$.  
As explained in Section~\ref{ss:Fr-cohomology}, the study of this question was inspired by a conjecture due to Humphreys~\cite{hum:cmr} (originally stated in the reductive case, but which admits a straightforward translation to the quantum case), which indeed follows from Theorem~\ref{thm:bezru-coh}.

\begin{thm}[\cite{bezru2} -- Humphreys conjecture]\label{thm:humphreys-q}
Suppose we are in the quantum case, and let $\lambda \in \bX^+$.  If $\lambda \in \lcalc(w)$ for some $w \in \fWext$, then
\[
V_\sg(\tilt(\lambda)) = \overline{\scO_\bc},
\]
where $\bc \subset \fWext$ is the antispherical $0$-cell containing $w$.
\end{thm}

\begin{rmk}
  Theorem~\ref{thm:humphreys-q} is not explicitly stated in~\cite{bezru2}, but it can be deduced from the results of~\cite{bezru2} as explained in~\cite[\S8.4]{ahr1}.
\end{rmk}

\begin{rmk}
  In the special case of type $\mathbf{A}$, Theorem~\ref{thm:humphreys-q} was proved earlier (using more elementary methods) by Ostrik~\cite{ostrik1}.
\end{rmk}  

\subsection{The reductive case}

The main results in the reductive case, stated below, come from~\cite{hardesty, ahr1}. 

\begin{thm}[\cite{ahr1}]\label{thm:ahr-coh}
Suppose we are in the reductive case, and assume either that $\sG = \SL_n$, or that $p$ is sufficiently large.  Let $\mu \in -\bX^+$, and let $(\scO_\mu, \sigma_\mu) \in \Xi$ be the pair corresponding to $\mu$ under the Lusztig--Vogan bijection (see~\S\ref{ss:nilpotent}).
Then the coherent sheaf $\cohsgb(\tilt(w_\mu \cdot_\ell 0))$ is (set-theoretically) supported on $\overline{\scO_\mu}$.
\end{thm}

\begin{thm}[\cite{hardesty, ahr1} -- Humphreys conjecture]\label{thm:humphreys-r}
Suppose we are in the reductive case.  Assume either that $\sG = \SL_n$, or that $p$ is sufficiently large.  Let $\lambda \in \bX^+$.  If $\lambda \in \lcalc(w)$ for some $w \in \fWext$, then
\[
V_\sg(\tilt(\lambda)) = \overline{\scO_\bc},
\]
where $\bc \subset \fWext$ is the antispherical $0$-cell containing $w$.
\end{thm}

Compared to the results of Section~\ref{ss:cohom-q}, these theorems suffer from two deficits.  The first is that if $\sG$ is not $\SL_n$, then the conclusions only hold when $\ell$ is ``sufficiently large,'' i.e., greater than some unknown bound that depends on the root datum $R$.  The case of $\SL_n$ suggests that one should expect this unknown bound to just be the Coxeter number in general, and the results of~\cite{ahr1} show that in the setting of Theorems~\ref{thm:ahr-coh} and~\ref{thm:humphreys-r} we always have $\overline{\scO_\mu} \subset \supp(\cohsgb(\tilt(w_\mu \cdot_\ell 0)))$ and $\overline{\scO_\bc} \subset V_\sg(\tilt(\lambda))$. But our current methods do not allow us to go further in this direction.

The second deficit is that Theorem~\ref{thm:ahr-coh}, unlike Theorem~\ref{thm:bezru-coh}, gives no description of $\cohsgb(\tilt(w_\mu \cdot_\ell 0))|_{\scO_\mu}$, nor does it involve the ``$\sigma_\mu$'' component in the Lusztig--Vogan bijection.  Already in $\SL_2$, one can check that this restriction need not be an irreducible vector bundle, so Theorem~\ref{thm:bezru-coh} cannot be copied verbatim in the reductive case.

\begin{prob}\label{prob:cohom}
Suppose we are in the reductive case.  Let $\mu \in -\bX^+$, and let $\bc \subset \fWext$ be the antispherical $0$-cell containing $w_\mu$.  What is the vector bundle $\cohsgb(\tilt(w_\mu \cdot_\ell 0))|_{\scO_\bc}$?  Can it be described in terms of the Lusztig--Vogan bijection?
\end{prob}

\section{Tensor ideals of tilting modules}

\subsection{Tensor ideals of tilting \texorpdfstring{$\sG$}{G}-modules}
\label{ss:tiltideal}

Let $\Tilt(\sG)$ denote the additive category of tilting modules for $\sG$, and let $[\Tilt(\sG)]$ denote its split Grothendieck group.  This is a free abelian group with a basis given by the classes of the indecomposable tilting modules
\[
\{ [\tilt(\lambda)] \mid \lambda \in \bX^+ \}.
\]
Since the tensor product of two tilting modules is again tilting (see~\cite[\S E.7]{jantzen} for details and references), the abelian group $[\Tilt(\sG)]$ acquires the structure of a (commutative) ring.  An ideal $I \subset [\Tilt(\sG)]$ is called a \emph{based ideal} if it is spanned (over $\Z$) by the classes of indecomposable tilting modules that it contains.  Given $m \in [\Tilt(\sG)]$, the \emph{principal based ideal} generated by $m$ is the unique minimal based ideal containing $m$.  (In general, this is larger than the usual principal ideal $(m)$ generated by $m$.) Following Ostrik and Andersen (see~\cite{ostrik, hha:pcells}), we define an equivalence relation on $\bX^+$ as follows:
\[
\lambda \Tsim \mu
\qquad
\begin{array}{@{}c@{}}
\text{if $[\tilt(\lambda)]$ and $[\tilt(\mu)]$ generate the same}\\
\text{principal based ideal of $[\Tilt(\sG)]$.}
\end{array}
\]
Equivalence classes for $\Tsim$ are called \emph{weight cells}.  In this section, we consider the problem of determining the weight cells.

\subsection{Determination of weight cells}

The following result relates weight cells to alcoves and antispherical $p$-cells.

\begin{thm}[\cite{ostrik, ahr1}]\label{thm:weight}
Let $\lambda, \mu \in \bX^+$.  Suppose $\lambda \in \lcalc(w)$ and $\mu \in \lcalc(v)$, where $w, v \in \fWext$.  Then $\lambda \Tsim \mu$ if and only if $w \sim^p v$.
\end{thm}

In the quantum ($p = 0$) case, this is due to Ostrik~\cite{ostrik}; in the reductive case, it was proved by the authors in~\cite[\S7]{ahr1}. The proofs in both cases are very similar: the main ingredient is a character formula for indecomposable tilting modules of the form $\tilt(w \cdot_\ell 0)$; see~\cite{soergel} in the quantum case, and~\cite{amrw} in the reductive case.

This statement has the appealing property that it treats the quantum and reductive cases uniformly.  Nevertheless, the two cases are qualitatively rather different.  In the quantum case, Theorem~\ref{thm:lusztig} implies that there are only finitely many weight cells, and that there is a canonical bijection
\[
\{\text{quantum weight cells} \} \simbij \{ \text{$\dot G$-orbits in $\cN$} \}.
\]
Moreover, thanks to Theorem~\ref{thm:humphreys-q}, this bijection can be realized in a concrete way: the $\dot G$-orbit corresponding to a quantum weight cell is the orbit whose closure is the support variety of any tilting module in that weight cell.  Thus, Theorems~\ref{thm:humphreys-q} and~\ref{thm:weight} can be rephrased as follows:

\begin{thm}\label{thm:supp-cell-q}
Suppose we are in the quantum case, and let $\lambda, \mu \in \bX^+$.  We have $\lambda \Tsim \mu$ if and only if $V_\sg(\tilt(\lambda)) = V_\sg(\tilt(\mu))$.
\end{thm}

In contrast, in the reductive case, there are always infinitely many weight cells (unless $\sG$ is a torus), but still only finitely many nilpotent orbits, so Theorem~\ref{thm:supp-cell-q} cannot possibly hold.  Indeed, by Theorem~\ref{thm:humphreys-r}, support varieties in the reductive case are still classified by $0$-cells (at least provided the Humphreys conjecture holds), and there is no obvious relationship between these and the $p$-cells appearing in Theorem~\ref{thm:weight}.

\begin{prob}\label{prob:cells}
In the reductive case, what is the relationship between weight cells and support varieties?  
(If the Humphreys conjecture holds, this is equivalent to understanding the relationship between antispherical $0$-cells and antispherical $p$-cells.)
\end{prob}

\section{Tilting vector bundles}

In this section, we propose conjectural solutions to Problems~\ref{prob:cohom} and~\ref{prob:cells}, placing both of them in the context of ``tilting vector bundles on nilpotent orbits.''  Before stating the conjectures, we need a generalization of the ideas from Section~\ref{ss:tiltideal}.

\subsection{Tilting modules for disconnected reductive groups}
\label{ss:tilt-disconn}

Let $H$ be a (possibly disconnected) algebraic group over $\bk$ whose identity component $H^\circ$ is a reductive group. (Such a group will be called a \emph{disconnected reductive group}.)  Assume furthermore that the characteristic $p$ of $\bk$ does not divide the order of the finite group $H/H^\circ$.  Then we may apply the theory of~\cite[\S3]{ahr2} to this group.  In particular, 
by~\cite[Theorem~3.7]{ahr2}, it makes sense to speak of standard, costandard, and tilting $H$-modules.  These classes of modules have a combinatorial parametrization which we will not recall here, except to note that the same combinatorics parametrizes \emph{irreducible} $H$-modules (see~\cite[Theorem~2.16]{ahr2}).  Thus, if $\Lambda$ is a set that parametrizes the set of irreducible $H$-modules, then we can associate to each $\sigma \in \Lambda$ an indecomposable tilting module $\tilt(\sigma)$.

\begin{lem}
The tensor product of two tilting $H$-modules is again tilting.
\end{lem}
Of course, this fact is well known for connected reductive groups (see Section~\ref{ss:tiltideal}), and the general case will be reduced to this one.
\begin{proof}
Let $M$ and $M'$ be two modules that admit a filtration by standard modules.  We will prove that $M \otimes M'$ admits a filtration by standard modules.  A very similar argument applies to modules with a filtration by costandard modules, and then the lemma follows by combining these two statements.

Recall that for both $H$- and $H^\circ$-modules (or for any highest-weight category), a module $X$ admits a filtration by standard modules if and only if the functor $\Ext^1(X,{-})$ vanishes on all costandard modules.  Thus, it is enough to show that $\Ext^1_H(M \otimes M',N) = 0$ for any costandard $H$-module $N$.  As explained in the proof of~\cite[Lemma~2.18]{ahr2}, there is a natural isomorphism
\[
\Ext^1_H(M \otimes M', N) \cong \left(\Ext^1_{H^\circ}(M \otimes M', N)\right){}^{H/H^\circ}.
\]
According to~\cite[Eq.~(3.3)]{ahr2}, any standard $H$-module regarded as an $H^\circ$-module is a direct sum of standard $H^\circ$-modules, and likewise for costandard modules.  It follows immediately that $\Ext^1_{H^\circ}(M \otimes M', N) = 0$, so we are done.
\end{proof}

Let $\Tilt(H)$ be the additive category of tilting $H$-modules, and let $[\Tilt(H)]$ be its split Grothendieck group.  This is a free abelian group with basis $\{ [\tilt(\sigma)] \mid \sigma \in \Lambda \}$, where $\Lambda$ is some set parametrizing the irreducible $H$-modules.  The discussion in Section~\ref{ss:tiltideal} can be repeated verbatim in this setting.  In particular, we may define an equivalence relation $\Tsim_H$ on $\Lambda$ by setting
\[
\sigma \Tsim_H \tau
\qquad
\begin{array}{@{}c@{}}
\text{if $[\tilt(\sigma)]$ and $[\tilt(\tau)]$ generate the same}\\
\text{principal based ideal of $[\Tilt(H)]$.}
\end{array}
\]

\begin{rmk}
 It is likely that the weight cells for $H$ and $H^\circ$ are closely related. However, at this point we have not been able to make this idea precise, and we do not have any analogue of Theorem~\ref{thm:weight} for disconnected reductive groups.
\end{rmk}

\subsection{Tilting vector bundles on nilpotent orbits}

Let $\scO \subset \cN$ be a nilpotent orbit.  Choose a point $x \in \scO$, and let $\dot G^x$ be its stabilizer.  Under our assumptions, the natural map $\dot G/\dot G^x \to \scO$ is an isomorphism of varieties, and there is therefore an equivalence of categories
\begin{equation}\label{eqn:coh-rep-equiv}
\Coh^{\dot G}(\scO) \cong \Rep(\dot G^x).
\end{equation}
Let $\dot G^x_\unip$ be the unipotent radical of the identity component $(\dot G^x)^\circ$, and let $\dot G^x_\red = \dot G_x/\dot G^x_\unip$.  Then $\dot G^x_\red$ is a disconnected reductive group.  Our assumptions imply that $p$ does not divide the order of the finite group $\dot G^x/(\dot G^x)^\circ = \dot G^x_\red/(\dot G^x_\red)^\circ$, so the discussion of Section~\ref{ss:tilt-disconn} can be applied to it.

Any $\dot G^x_\red$-representation $M$ gives rise to coherent sheaf on $\scO$ by first inflating to $\dot G^x$ and then passing through~\eqref{eqn:coh-rep-equiv}.  In particular, every irreducible vector bundle $\scL_\scO(\sigma)$ on $\scO$ arises in this way from some irreducible $\dot G^x_\red$-representation, which we denote by $\irr_{\dot G^x_\red}(\sigma)$.  In this way, the set $\Sigma_\scO$ parametrizes the set of isomorphism classes of irreducible $\dot G^x_\red$-representations.

Each $\sigma \in \Sigma_\scO$ also gives rise to an indecomposable tilting $\dot G^x_\red$-representation $\tilt_{\dot G^x_\red}(\sigma)$, and hence to a vector bundle on $\scO$ that we denote by
$
\scT_\scO(\sigma).
$
Vector bundles obtained in this way are called \emph{tilting vector bundles}.  We define an equivalence relation on $\Sigma_\scO$ by setting
\[
\sigma \Tsim_\scO \tau
\qquad
\begin{array}{@{}c@{}}
\text{if $[\tilt_{\dot G^x_\red}(\sigma)]$ and $[\tilt_{\dot G^x_\red}(\tau)]$ generate the same}\\
\text{principal based ideal of $[\Tilt(\dot G^x_\red)]$.}
\end{array}
\]
This equivalence relation on $\Sigma_\scO$ is easily seen to be independent of the choice of $x$.
Equivalence classes for $\Tsim_\scO$ are called \emph{vector bundle cells}.

\subsection{Conjectures}

In this section we present some conjectures that might help place the Humphreys conjecture in a larger context, and facilitate its proof in full generality. We expect these conjectures to hold under the assumptions of the present paper (namely, as soon as $p>h$), but stronger conditions might be needed.

The following statement is a first step towards Problem~\ref{prob:cells}. (It can be 
regarded as a reformulation of \cite[Corollary~17]{hha:pcells} in the modern language of $p$-cells.)

\begin{prop}\label{prop:pcell-0cell}
Assume that Theorem~\ref{thm:humphreys-r} holds for our data,
and let $w, v \in \fWext$.  If $w \sim^p v$, then $w \sim^0 v$.  Thus, every antispherical $0$-cell is a union of antispherical $p$-cells.
\end{prop}

\begin{proof}
By Theorem~\ref{thm:weight}, if $w \sim^p v$, then $w \cdot_\ell 0$ and $v \cdot_\ell 0$ are in the same weight cell. This means that there exist tilting modules $T'$ and $T''$ such that
\begin{gather*}
\text{$\tilt(w \cdot_\ell 0)$ is a direct summand of $\tilt(v \cdot_\ell 0) \otimes T'$, and} \\
\text{$\tilt(v \cdot_\ell 0)$ is a direct summand of $\tilt(w \cdot_\ell 0) \otimes T''$.}
\end{gather*}
By general properties of support varieties (see \cite[\S 8.1]{ahr1}), these conditions imply that $V_\sg(\tilt(w \cdot_\ell 0)) = V_\sg(\tilt(v \cdot_\ell 0))$. Theorem~\ref{thm:humphreys-r} then tells us that $w \sim^0 v$.
\end{proof}

\begin{rmk}
 It is tempting to conjecture that in any crystallographic Coxeter group $0$-cells always decompose into $p$-cells. However, a counterexample to this phenomenon has been observed by Jensen~\cite{jensen} (for type $\mathbf{C}_3$ in characteristic $2$).
\end{rmk}

Below, we propose a refinement of Proposition~\ref{prop:pcell-0cell} that makes it more precise how cells decompose.  To state it, we need the transfer of the vector bundle cell relation across the Lusztig--Vogan bijection.  For $\lambda, \mu \in -\bX^+$, we set
\[
\lambda \LVsim \mu 
\qquad\text{if $\scO_\lambda = \scO_\mu$ and $\sigma_\lambda \Tsim_{\scO_\lambda} \sigma_\mu$.}
\]
We define the \emph{Lusztig--Vogan cells} in $-\bX^+$ to be the equivalence classes for $\LVsim$ . By definition, there is a bijection
\[
\{ \text{Lusztig--Vogan cells} \} 
\simbij
\left\{ (\scO, \bb) \,\Big|\,
\begin{array}{@{}c@{}}
\text{$\scO \subset \cN$ a nilpotent orbit, and} \\
\text{$\bb \subset \Sigma_\scO$ a vector bundle cell}
\end{array}
\right\}.
\]

\begin{conj}\label{conj:pcell-class}
Let $\lambda, \mu \in -\bX^+$, and consider $w_\lambda, w_\mu \in \fWfext$.  We have
\[
\lambda \LVsim \mu
\qquad\text{if and only if}\qquad
w_\lambda \sim^p w_\mu.
\]
\end{conj}

As an immediate consequence, we obtain the following statement, which can be seen as a (conjectural) generalization of Theorem~\ref{thm:lusztig}.

\begin{prop}\label{prop:pcell-class}
If Conjecture~\ref{conj:pcell-class} holds, there is a canonical bijection
\[
\{\text{antispherical $p$-cells} \}
\simbij
\left\{ (\scO, \bb) \,\Big|\,
\begin{array}{@{}c@{}}
\text{$\scO \subset \cN$ a nilpotent orbit, and} \\
\text{$\bb \subset \Sigma_\scO$ a vector bundle cell}
\end{array}
\right\}.
\]
\end{prop}
\begin{proof}
The proposition follows immediately from the fact that every antispherical $p$-cell in $\fWext$ meets $\fWfext$ 
(see \cite[Lemma~8.8]{ahr1}). 
\end{proof}
It is useful to observe that this proposed description of the antispherical $p$-cells is ``recursive,''
in the sense that it is determined by the description of the antispherical $p$-cells for $\dot G^x_\red$
(with the caveat that the groups $\dot G^x_\red$ are sometimes \emph{disconnected} reductive groups). The recursive nature is best illustrated in the case of the general linear group, which we will consider in
\S\ref{ss:gen-lin-comb}.

Finally, the following conjecture (which would provide a more faithful analogue of Theorem~\ref{thm:bezru-coh} in the reductive case) would solve Problem~\ref{prob:cohom}.

\begin{conj}\label{conj:cohom}
Suppose we are in the reductive case.  Let $\mu \in -\bX^+$, and let $(\scO_\mu, \sigma_\mu) \in \Xi$ be the pair corresponding to $\mu$ under the Lusztig--Vogan bijection. The coherent sheaf $\cohsgb(\tilt(w_\mu \cdot_\ell 0))$ is scheme-theoretically supported on $\overline{\scO_\mu}$, and
\[
\cohsgb(\tilt(w_\mu \cdot_\ell 0))|_{\scO_\mu} \cong \scT_{\scO_\mu}(\sigma_\mu).
\]
\end{conj}

\begin{rmk}
 Theorem~\ref{thm:bezru-coh} was proved by giving an even more precise description of $\cohsgb(\tilt(w_\mu \cdot_\ell 0))$, as the cohomology of a perverse-coherent sheaf on $\cN$. Unfortunately, at this point we do not have any conjecture for such a description in the reductive case.
\end{rmk}

These conjectures yield a partial generalization of Theorem~\ref{thm:supp-cell-q}, as follows.

\begin{prop}
Suppose that Theorem~\ref{thm:humphreys-r} holds for our data, and that Conjectures~\ref{conj:pcell-class} and \ref{conj:cohom} hold. Let $\lambda, \mu \in -\bX^+$. We have $w_\lambda \cdot_\ell 0 \Tsim w_\mu \cdot_\ell 0$ if and only if the following conditions both hold:
\begin{enumerate}
\item $V_\sg(\tilt(w_\lambda \cdot_\ell 0)) = V_\sg(\tilt(w_\mu \cdot_\ell 0))$.
\item Let $\scO$ be the dense orbit in the support varieties mentioned above.  Then the tilting vector bundles $\cohsgb(\tilt(w_\lambda \cdot_\ell 0))|_\scO$ and $\cohsgb(\tilt(w_\mu \cdot_\ell 0))|_\scO$ belong to the same vector bundle cell.
\end{enumerate}
\end{prop}

\section{Combinatorics for general linear groups}
\label{ss:gen-lin-comb}

Assume that Conjecture~\ref{conj:pcell-class} holds, and suppose $\sG$ is a product of general linear groups.  It is well known that the reductive quotient of the stabilizer of any nilpotent element in $\cN$ is again a product of general linear groups.  Then, according to Proposition~\ref{prop:pcell-class}, weight cells for $\sG$ are in bijection with pairs of the form
$(\scO, \bb)$ where
$\scO \subset \cN$ is a nilpotent orbit for $\sG$, and $\bb$ is a weight cell for a (usually smaller) product of general linear groups.
But now we can iterate the process: $\bb$ is again classified by pairs consisting of a nilpotent orbit and a weight cell for a (usually smaller) product of general linear groups.

In this section we explain how to record this procedure with combinatorics. For $\sa = (a_1, \ldots, a_k)$ a $k$-tuple of positive integers, we set
\[
\GL_\sa = \GL_{a_1} \times \cdots \times \GL_{a_k},
\]
and choose in the standard way the maximal torus and (negative) Borel subgroup. Of course, this group identifies in a canonical way with its Frobenius twist.

We define a \emph{multipartition} of $\sa$ to be a $k$-tuple of partitions $\pi = (\pi_1, \ldots, \pi_k)$, where $\pi_i \vdash a_i$.  To indicate that $\pi$ is a multipartition of $\sa$, we write
\[
\pi \vDash \sa.
\]
Multipartitions of $\sa$ parametrize nilpotent orbits in $\GL_\sa$: for $\pi \vDash \sa$ we denote by $x_\pi$ a fixed nilpotent matrix with Jordan blocks of sizes determined by $\pi$.

Next, write each consituent partition of $\pi$ as a list of parts, with multiplicities given as exponents: say
\[
\pi_i = [\pi_{i,1}^{b_{i,1}}, \pi_{i,2}^{b_{i,2}}, \ldots, \pi_{i,m_i}^{b_{i,m_i}}]
\]
where $\pi_{i,1} > \pi_{i,2} > \cdots > \pi_{i,m_i} > 0$, and $b_{i,j} \ge 1$ for all $j$. The list of all multiplicities in $\pi$ is denoted by
\[
\mult(\pi) = (b_{1,1}, b_{1,2}, \ldots, b_{1,m_1}, b_{2,1}, b_{2,2}, \ldots, b_{2,m_2}, \ldots, b_{k,1}, b_{k,2},\ldots, b_{k,m_k}).
\]
There there is an isomorphism
\[
(\GL_\sa)^{x_\pi}_\red \cong \GL_{\mult(\pi)}.
\]
Fix such an isomorphism for each multipartition $\pi$. (An explicit choice of isomorphism is given e.g.~in~\cite[\S 3.8]{jantzen-nilp}.)
Then Proposition~\ref{prop:pcell-class} says that there is a bijection
\[
\{\text{weight cells for $\GL_\sa$}\} 
\simbij
\left\{
(\pi, \bb) \,\Big|\,
\begin{array}{@{}c@{}}
\text{$\pi \vDash \sa$, and $\bb$ is a weight}\\
\text{cell for $\GL_{\mult(\pi)}$}
\end{array}\right\}.
\]
We can iterate this process to obtain the following result.

\begin{prop}
Suppose Conjecture~\ref{conj:pcell-class} 
holds, and let $\sa = (a_1, \ldots, a_k)$ be a $k$-tuple of positive integers.  There is a bijection
\[
\{ †\text{antispherical $p$-cells for $\GL_\sa$} \} 
\simbij
\{ (\pi^{(1)}, \pi^{(2)}, \ldots ) \},
\]
where the right-hand side is the set of sequences of multipartitions satisfying the following conditions:
\begin{enumerate}
\item $\pi^{(1)} \vDash \sa$, and for each $j > 1$, $\pi^{(j)} \vDash \mult(\pi^{(j-1)})$.\label{it:mpdash}
\item For $j$ sufficiently large, we have $\mult(\pi^{(j)}) = (1,1, \ldots, 1)$.\label{it:evtriv}
\end{enumerate}
\end{prop}

The only multipartition of $(1,1,\ldots, 1)$ is $([1],[1], \ldots, [1])$, so as soon as condition~\eqref{it:evtriv} above holds, all later terms in the sequence must be trivial. A sequence of multipartitions satisfying condition~\eqref{it:evtriv} above will be said to be \emph{eventually trivial}.

\begin{proof}
Let $\cP$ be the set of all sequences of multipartitions satisfying condition~\eqref{it:mpdash} in the statement above, and let $\cP_0 \subset \cP$ be the subset consisting of eventually trivial sequences.  In the discussion preceding the proposition, we have constructed a map
\[
\gamma: \{ \text{antispherical $p$-cells} \} \to \cP;
\]
more precisely, to an antispherical $p$-cell $\bc$ we associate the multipartition $\pi$ corresponding to the orbit attached to any element $w \in \fWfext \cap \bc$ under the Lusztig--Vogan bijection (this intersection is nonempty, as explained in the proof of Proposition~\ref{prop:pcell-class}), then repeat the procedure for the antispherical $p$-cell corresponding to the weight cell in $(\GL_\sa)^{x_\pi}_\red \cong \GL_{\mult(\pi)}$ containing the simple module attached to $w$ (see Theorem~\ref{thm:weight}), and continue.

Let us first show that $\gamma$ takes values in $\cP_0$.  Let $\bc$ be an antispherical $p$-cell, and suppose $\gamma(\bc) = (\pi^{(1)}, \pi^{(2)}, \ldots)$.  Let $\sa^0 = \sa$, and for $j \ge 1$, let $\sa^j = \mult(\pi^{(j)})$.  Suppose $\sa^j = (a^j_1, \ldots, a^j_{m_j})$, and then let
\[
r^j = \sum_{s=1}^{m_j} (a^j_s - 1).
\]
The integer $r^j$ is the semisimple rank of the group $\GL_{\sa^j}$, which in turn is the reductive quotient of the centralizer of a nilpotent element in the Lie algebra of $\GL_{\sa^{j-1}}$.  Thus, $r^j$ is a weakly decreasing function of $j$.  If $r^j$ ever becomes $0$, then $\GL_{\sa^j}$ is a torus, and the sequence $\gamma(\bc)$ is trivial from then on.  It is easily checked that if $\pi^{(j)}$ does not correspond to the zero nilpotent orbit in $\GL_{\sa^{j-1}}$, then $r^j < r^{j-1}$.

Thus, if $\gamma(\bc)$ is \emph{not} eventually trivial, then after discarding finitely many terms at the beginning of the sequence, we may assume that all the multipartitions $\pi^{(j)}$ correspond to the zero orbit.  This implies that $\sa^j = \sa$ for all $j$.  Let $\scO_0$ denote the zero nilpotent orbit for $\GL_\sa$.

Choose (as we may) a weight $\mu \in -\bX^+$ such that $w_\mu \in \bc \cap \fWfext$.  Since $\mu$ corresponds under the Lusztig--Vogan bijection to a pair involving $\scO_0$, it follows from~\cite[Proposition~4.9 and its proof]{achar2} (see also~\cite[Proposition~6]{hha:pcells}) that $\mu \in -\bX^+ - 2\rho$, and that it corresponds under the Lusztig--Vogan bijection to $(\scO_0, w_0\mu - 2\rho)$. (Here $w_0$ is the longest element in $\Wf$, and we parametrize simple equivariant vector bundles on $\scO_0$, i.e.~simple $\GL_\sa$-modules, by their highest weight.) We may replace the weight $w_0\mu - 2\rho$ by any other weight in its weight cell, at the expense of making a corresponding change in $\mu$.  Thus, we may assume without loss of generality that $w_0\mu - 2\rho \in \fWfext \cdot_\ell 0$, say $w_0\mu - 2\rho = w_{\mu_1} \cdot_\ell 0$ for some $\mu_1 \in -\bX^+$.  Since $\mu_1$ again corresponds under the Lusztig--Vogan bijection to a pair involving $\scO_0$, we can repeat this reasoning as many times as we wish.  If we repeat it $m$ times, we obtain a sequence of elements $\mu_0 = \mu, \mu_1, \mu_2, \ldots, \mu_m \in -\bX^+$ that satisfy
\[
w_0\mu_{i-1} - 2\rho = w_{\mu_i} \cdot_\ell 0
\qquad\text{and}\qquad
\mu_i \in -\bX^+ - 2\rho\qquad\text{for $i = 1, \ldots, m$.}
\]
The latter condition implies that $w_{\mu_{i}} \cdot_\ell 0 = w_0(\ell\mu_{i} + \rho) - \rho = \ell w_0\mu_{i} - 2\rho$, and the former condition then tells us that $w_0\mu_{i-1} - 2\rho = w_{\mu_{i}} \cdot_\ell 0 = \ell w_0\mu_{i} - 2\rho$, and hence
\[
\mu_{i-1} = \ell\mu_i \qquad\text{for $i = 1, \ldots, m$.}
\]
It follows that $\mu \in -\ell^m\bX^+$.  By letting $m$ vary, we see that $\bc$ contains elements $w_\mu$ with $\mu \in \ell^m\bX \cap (-\bX^+ -2\rho)$ for any $m$.  For such $\mu$, for any simple coroot $\alpha^\vee$ we have $\la w_0\mu, \alpha^\vee \ra \ge \ell^m$, and hence
\[
\la w_\mu \cdot_\ell 0, \alpha^\vee \ra = \la \ell w_0\mu - 2\rho, \alpha^\vee \ra \ge \ell^{m+1} - 2 \geq \ell^{m} - 1.
\]
In view of~\cite[Lemma~E.8]{jantzen}, this implies that $\tilt(w_\mu \cdot_\ell 0)$ is projective for the $m$-th Frobenius kernel $\sG_{m}$.

Finally we observe that if $\nu, \nu' \in \bX^+$ are in the same weight cell, then $\tilt(\nu)$ is projective over $\sG_{r}$ if and only if $\tilt(\nu')$ is.  The considerations in the preceding paragraph show that for all $\nu$ in the weight cell corresponding to $\bc$, $\tilt(\nu)$ is projective over $\sG_{m}$ for all $m \ge 1$. But~\cite[Lemma~E.8]{jantzen} also implies that a given $\tilt(\nu)$ can be projective over $\sG_{r}$ only for finitely many $r$, so we have a contradiction.

Now that we have shown that $\gamma$ takes values in $\cP^0$, a straightforward induction argument on the number of nontrivial terms in $\gamma(\bc)$ shows that this map is both injective and surjective.
\end{proof}

Table~\ref{tab:cell} shows the sequences of multipartitions for a few small general linear groups.  In this table, the point at which the sequence of multipartitions becomes trivial is indicated by a bullet ($\bullet$).

\begin{table}[t]
\[
\hbox{\small$\begin{array}[t]{l}
\text{Weight cells for $\GL_2$} \\
\hline
(([2]), \bullet) \\
(([1^2]), ([2]), \bullet) \\
(([1^2]), ([1^2]), ([2]), \bullet) \\
\qquad \vdots \\
\ \\
\ \\
\text{Weight cells for $\GL_3$} \\
\hline
(([3]), \bullet) \\
(([2,1]), \bullet) \\
(([1^3]), ([3]), \bullet) \\
(([1^3]), ([2,1]), \bullet) \\
(([1^3]), ([1^3]), ([3]), \bullet) \\
(([1^3]), ([1^3]), ([2,1]), \bullet) \\
\qquad\vdots
\end{array}
\qquad
\begin{array}[t]{l}
\text{Weight cells for $\GL_4$} \\
\hline
(([4]), \bullet) \\
(([3,1]), \bullet) \\
(([2^2]), ([2]), \bullet) \\
(([2^2]), ([1^2]), ([2]), \bullet) \\
(([2^2]), ([1^2]), ([1^2]), ([2]), \bullet) \\
\qquad \vdots \\
(([2,1^2]), ([1],[2]), \bullet) \\
(([2,1^2]), ([1],[1^2]), ([1],[2]), \bullet) \\
(([2,1^2]), ([1],[1^2]), ([1],[1^2]), ([1],[2]), \bullet) \\
\qquad \vdots \\
(([1^4]), ([4]), \bullet) \\
(([1^4]), ([3,1]), \bullet) \\
(([1^4]), ([2^2]), ([2]), \bullet) \\
(([1^4]), ([2^2]), ([1^2]), ([2]), \bullet) \\
(([1^4]), ([2^2]), ([1^2]), ([1^2]), ([2]), \bullet) \\
\qquad \vdots \\
(([1^4]), ([2,1^2]), ([1],[2]), \bullet) \\
(([1^4]), ([2,1^2]), ([1],[1^2]), ([1],[2]), \bullet) \\
(([1^4]), ([2,1^2]), ([1],[1^2]), ([1],[1^2]), ([1],[2]), \bullet) \\
\qquad \vdots
\end{array}$}
\]
\caption{Weight cells in general linear groups}\label{tab:cell}
\end{table}

\section{Evidence for the conjectures}

In case Theorem~\ref{thm:humphreys-r} holds, Proposition~\ref{prop:pcell-0cell} and Theorem~\ref{thm:lusztig} together give us a map
\[
\{\text{antispherical $p$-cells}\} \to \{ \text{nilpotent orbits} \}.
\]
In this section, we will consider a few cases where this map can be explicitly upgraded to a bijection of the form considered in Proposition~\ref{prop:pcell-class}.

\subsection{The regular orbit}

Let $\scO_\reg$ be the regular nilpotent orbit for $\dot G$.  If $x \in \scO_\reg$, then $\dot G^x_\red$ can be identified with the center of $\dot G$.  From this, one can see that there is a unique vector bundle cell on $\scO_\reg$, so we need to check that the corresponding antispherical $0$-cell (namely, $\Omega$) consists of a single antispherical $p$-cell. This follows from~\cite[Example~5.5]{ahr1}. 

\subsection{The subregular orbit}

Assume that $\sG$ is quasi-simple and not of type $\mathbf{A}_1$, and let $\scO_\sreg$ be the subregular nilpotent orbit.  One can check case-by-case that for $x \in \scO_\sreg$, the identity component $(\dot G^x_\red)^\circ$ of the centralizer is a torus, so there is a unique vector bundle cell on this orbit.  We must again check that the corresponding antispherical $0$-cell is also an antispherical $p$-cell.  If $\sG$ is either simply laced or of type $\mathbf{G}_2$, this can be be deduced from the results of~\cite{ras:sctm}.  (It should be possible to prove it by similar methods in the remaining cases.)

\subsection{The zero orbit}

Let $\bc_0$ be the antispherical $0$-cell corresponding to the zero nilpotent orbit $\scO_0$.
Since the centralizer of the zero nilpotent element is isomorphic to $\sG$, the part of Proposition~\ref{prop:pcell-class} corresponding to $\scO_0$ can be rephrased as a bijection
\[
\{ \text{antispherical $p$-cells contained in $\bc_0$} \}
\simbij
\{ \text{weight cells for $\sG$} \}.
\]
Such a bijection can be deduced from~\cite[Proposition~14]{hha:pcells}.

\subsection{Some low rank cases}

If $\sG$ is of type $\mathbf{A}_1$ or $\mathbf{A}_2$, the preceding paragraphs yield Proposition~\ref{prop:pcell-class} in full.

If $\sG$ is of type $\mathbf{B}_2$, there is one remaining nilpotent orbit.  Proposition~\ref{prop:pcell-class} can be established in this case by combining the explicit determination of the Lusztig--Vogan bijection (see~\cite[Appendix B]{a:phd}) with the description of the corresponding $p$-cells given in~\cite[Corollary~10.32]{jensen}.



\begin{thebibliography}{AMRW}

\bibitem[A1]{a:phd}
P.~Achar, {\em Equivariant coherent sheaves on the nilpotent cone for complex
  reductive Lie groups}, Ph.D. thesis, Massachusetts Institute of Technology,
  2001.

\bibitem[A2]{achar}
P.~Achar, {\em Perverse coherent sheaves on the nilpotent cone in good
  characteristic}, Recent developments in Lie algebras, groups and
  representation theory, Proc. Sympos. Pure Math., vol.~86, Amer. Math. Soc.,
  2012, pp.~1--23.
  
\bibitem[A3]{achar2}
P.~Achar, \emph{On exotic and perverse-coherent sheaves}, in \emph{Representations of reductive groups}, 11--49,
Progr. Math. 312, Birkh\"auser/Springer, 2015. 
  
\bibitem[AHR1]{ahr1}
P.~Achar, W.~Hardesty, and S.~Riche, {\em On the Humphreys conjecture on
  support varieties of tilting modules}, Transform. Groups, to appear,
  \href{https://arxiv.org/abs/1707.07740}{arXiv:1707.07740}.

\bibitem[AHR2]{ahr2}
P.~Achar, W.~Hardesty, and S.~Riche, {\em Representation theory of disconnected
  reductive groups}, \href{https://arxiv.org/abs/1810.06851}{arXiv:1810.06851}.

\bibitem[AHR3]{ahr3}
P.~Achar, W.~Hardesty, and S.~Riche, {\em Integral exotic sheaves and the
  modular Lusztig--Vogan bijection}, \href{https://arxiv.org/abs/1810.08897}{arXiv:1810.08897}.
  
\bibitem[AMRW]{amrw}
P.~Achar, S.~Makisumi, S.~Riche, G.~Williamson, \emph{Koszul duality for Kac--Moody groups and characters of tilting modules},
J. Amer. Math. Soc. \textbf{32} (2019), 
261--310. 
  
 \bibitem[An]{hha:pcells}
H.~H. Andersen, {\em Cells in affine {W}eyl groups and tilting modules}, in
  \emph{Representation theory of algebraic groups and quantum groups}, 1--16, Adv. Stud. Pure
  Math. 40, Math. Soc. Japan, 
  2004.

\bibitem[AJ]{aj}
H.~H. Andersen and J.~C. Jantzen, {\em Cohomology of induced representations
  for algebraic groups}, Math. Ann. {\bf 269} (1984), 487--525.
  
  \bibitem[BNPP]{bnpp:qggnc}
C.~P. Bendel, D.~K. Nakano, B.~J. Parshall, and C.~Pillen, {\em Cohomology for
  quantum groups via the geometry of the nullcone}, Mem. Amer. Math. Soc. {\bf
  229} (2014), no.~1077, x+93.


\bibitem[B1]{bezru1}
R.~Bezrukavnikov, {\em Quasi-exceptional sets and equivariant coherent sheaves
  on the nilpotent cone}, Represent. Theory {\bf 7} (2003), 1--18.
  
\bibitem[B2]{bezru-cells}
 R.~Bezrukavnikov, \emph{On tensor categories attached to cells in affine Weyl groups}, in \emph{Representation theory of algebraic groups and quantum groups}, 69--90, Adv. Stud. Pure Math. 40, Math. Soc. Japan, 
 2004.

\bibitem[B3]{bezru2}
R.~Bezrukavnikov, {\em Cohomology of tilting modules over quantum groups and
  {$t$}-structures on derived categories of coherent sheaves}, Invent. Math.
  {\bf 166} (2006), 327--357.

\bibitem[B4]{bezru3}
R.~Bezrukavnikov, {\em Perverse sheaves on affine flags and nilpotent cone of
  the Langlands dual group}, Israel J. Math. {\bf 170} (2009), 185--206.

\bibitem[FP]{fp}
E.~Friedlander and B.~Parshall, {\em Cohomology of Lie algebras and algebraic
  groups}, Amer. J. Math. {\bf 108} (1986), 235--253.

\bibitem[GK]{gk}
V.~Ginzburg and S.~Kumar, {\em Cohomology of quantum groups at roots of unity},
  Duke Math. J. {\bf 69} (1993), 179--198.

\bibitem[Ha]{hardesty}
W.~Hardesty, {\em On support varieties and the Humphreys conjecture in type A},
  Adv. Math. {\bf 329} (2018), 392--421.
  
  \bibitem[Hu]{hum:cmr}
J.~E. Humphreys, {\em Comparing modular representations of semisimple groups
  and their Lie algebras}, Modular interfaces ({R}iverside, {CA}, 1995), AMS/IP
  Stud. Adv. Math., vol.~4, Amer. Math. Soc., Providence, RI, 1997, pp.~69--80.

\bibitem[J1]{jantzen}
J.~C. Jantzen, {\em Representations of algebraic groups}, 2nd ed., Mathematical
  Surveys and Monographs, no. 107, Amer. Math. Soc., Providence, RI, 2003.
  
\bibitem[J2]{jantzen-nilp}
J.~C.~Jantzen, \emph{Nilpotent orbits in representation theory}, in \emph{Lie theory}, 1--211,
Progr. Math. 228, Birkh\"auser Boston, 2004.

\bibitem[Je]{jensen}
L.~T. Jensen, {$p$-Kazhdan--Lusztig theory}, Ph.D. thesis, Universit\"at Bonn, 2018.

\bibitem[JW]{jw}
L.~T. Jensen and G.~Williamson, {\em The $p$-canonical basis for Hecke
  algebras}, Categorification and higher representation theory, Contemp. Math.,
  vol. 683, Amer. Math. Soc., Providence, RI, 2017, pp.~333--361.

\bibitem[KL]{kl}
D.~Kazhdan and G.~Lusztig, {\em Representations of Coxeter groups and Hecke
  algebras}, Invent. Math. {\bf 53} (1979), 165--184.

\bibitem[Lu]{lusztig}
G.~Lusztig, {\em Cells in affine {W}eyl groups. {IV}}, J. Fac. Sci. Univ. Tokyo
  Sect. IA Math. {\bf 36} (1989), 297--328.

\bibitem[LX]{lx}
G.~Lusztig and N.~Xi, {\em Canonical left cells in affine Weyl groups}, Adv.
  Math. {\bf 72} (1988), 284--288.

\bibitem[O1]{ostrik1}
V.~Ostrik, \emph{Cohomological supports for quantum groups},
Funktsional. Anal. i Prilozhen. \textbf{32} (1998), 
22--34, 95; translation in
Funct. Anal. Appl. \textbf{32} (1998), 
237--246 (1999). 

\bibitem[O2]{ostrik}
V.~Ostrik, {\em Tensor ideals in the category of tilting modules}, Transform.
  Groups {\bf 2} (1997), 279--287.

\bibitem[Ra]{ras:sctm}
T.~E. Rasmussen, {\em Multiplicities of second cell tilting modules}, J.
  Algebra {\bf 288} (2005), 
  1--19.

\bibitem[RW]{rw}
S.~Riche and G.~Williamson, {\em Tilting modules and the $p$-canonical basis},
  Ast\'erisque {\bf 397} (2018).

\bibitem[So]{soergel}
W.~Soergel, \emph{Character formulas for tilting modules over Kac--Moody algebras},
Represent. Theory \textbf{2} (1998), 432--448. 
  
\bibitem[V]{vogan}
D.~A. Vogan, {\em The method of coadjoint orbits for real reductive groups},
  Representation theory of {L}ie groups ({P}ark {C}ity, {UT}, 1998), IAS/Park
  City Math. Ser., vol.~8, Amer. Math. Soc., Providence, RI, 2000,
  pp.~179--238.

\end{thebibliography}
\end{document}